\documentclass[reqno]{amsart}
\usepackage{amssymb,setspace}
\usepackage[hmarginratio=1:1]{geometry}
\usepackage{ifpdf}
\ifpdf
 \usepackage[hyperindex]{hyperref}
\else
 \expandafter\ifx\csname dvipdfm\endcsname\relax
 \usepackage[hypertex,hyperindex]{hyperref}
 \else
 \usepackage[dvipdfm,hyperindex]{hyperref}
 \fi
\fi
\allowdisplaybreaks[4]
\numberwithin{equation}{section}
\theoremstyle{plain}
\newtheorem{thm}{Theorem}[section]
\newtheorem{cor}{Corollary}[thm]
\newtheorem{lem}{Lemma}[section]
\theoremstyle{remark}
\newtheorem{rem}{Remark}[section]
\DeclareMathOperator{\td}{d\mspace{-1mu}}

\begin{document}

\title[Hermite-Hadamard type integral inequalities]
{Hermite-Hadamard type integral inequalities for functions whose first derivatives are of convexity}

\author[F. Qi]{Feng Qi}
\address[Qi]{College of Mathematics, Inner Mongolia University for Nationalities, Tongliao City, Inner Mongolia Autonomous Region, 028043, China}
\email{\href{mailto: F. Qi <qifeng618@gmail.com>}{qifeng618@gmail.com}, \href{mailto: F. Qi <qifeng618@hotmail.com>}{qifeng618@hotmail.com}, \href{mailto: F. Qi <qifeng618@qq.com>}{qifeng618@qq.com}}
\urladdr{\url{http://qifeng618.wordpress.com}}

\author[T.-Y. Zhang]{Tian-Yu Zhang}
\address[Zhang]{College of Mathematics, Inner Mongolia University for Nationalities, Tongliao City, Inner Mongolia Autonomous Region, 028043, China}
\email{\href{mailto: T.-Y. Zhang <zhangtianyu7010@126.com>}{zhangtianyu7010@126.com}, \href{mailto: T.-Y. Zhang <zhangtian-yu@qq.com>}{zhangtian-yu@qq.com}}

\author[B.-Y. Xi]{Bo-Yan Xi}
\address[Xi]{College of Mathematics, Inner Mongolia University for Nationalities, Tongliao City, Inner Mongolia Autonomous Region, 028043, China}
\email{\href{mailto: B.-Y. Xi <baoyintu78@qq.com>}{baoyintu78@qq.com}, \href{mailto: B.-Y. Xi <baoyintu68@sohu.com>}{baoyintu68@sohu.com}, \href{mailto: B.-Y. Xi <baoyintu78@imun.edu.cn>}{baoyintu78@imun.edu.cn}}



\begin{abstract}
In the paper, the authors establish some new Hermite-Hadamard type inequalities for functions whose first derivatives are of convexity and apply these inequalities to construct inequalities of special means.
\end{abstract}

\thanks{This work was partially supported by the Foundation of the Research Program of Science and Technology at Universities of Inner Mongolia Autonomous Region under Grant No.~NJZY13159 and by the NSFC under Grant No.~10962004}

\thanks{This paper was typeset using\AmS-\LaTeX}

\maketitle

\section{Introduction}

In~\cite{Dragomir-Agarwal-AML-98-95}, the following Hermite-Hadamard type inequalities for differentiable convex functions were proved.

\begin{thm}[{\cite[Theorem~2.2]{Dragomir-Agarwal-AML-98-95}}]
Let $f:I^\circ\subseteq\mathbb{R}\to\mathbb{R}$ be a differentiable mapping on $I^\circ$ and $a,b\in I^\circ$ with $a<b$. If $\rvert f'(x)\rvert$ is convex on $[a,b]$, then
\begin{equation}
\biggl\lvert\frac{f(a)+f(b)}2-\frac1{b-a}\int_a^bf(x)\td x\biggr\rvert\le\frac{(b-a)(\lvert f'(a)\rvert +\lvert f'(b)\rvert )}8.
\end{equation}
\end{thm}

\begin{thm}[{\cite[Theorem~2.3]{Dragomir-Agarwal-AML-98-95}}]
Let $f:I^\circ\subseteq\mathbb{R}\to\mathbb{R}$ be a differentiable mapping on $I^\circ$, $a,b\in I^\circ$ with $a<b$, and $p>1$. If the new mapping $\lvert f'(x)\rvert^{p/(p-1)}$ is convex on $[a,b]$, then
\begin{equation}
\biggl\lvert\frac{f(a)+f(b)}2-\frac1{b-a}\int_a^bf(x)\td x\biggr\rvert
\le\frac{b-a}{2(p+1)^{1/p}} \biggl[\frac{\lvert f'(a)\rvert^{p/(p-1)}+\lvert f'(b)\rvert^{p/(p-1)}}2\biggr]^{(p-1)/p}.
\end{equation}
\end{thm}

In~\cite{Pearce-Pecaric-AML-00-55}, the above inequalities were generalized as follows.

\begin{thm}[{\cite[Theorems~1~and~2]{Pearce-Pecaric-AML-00-55}}]
Let $f:I\subseteq\mathbb{R}\to\mathbb{R}$ be differentiable on $I^\circ$, $a,b\in I$ with $a<b$, and $q\ge1$. If $\lvert f'(x)\rvert^q$ is convex on $[a,b]$, then
\begin{equation}
\biggl\lvert\frac{f(a)+f(b)}2-\frac1{b-a}\int_a^bf(x)\td x\biggr\rvert\le\frac{b-a}4 \biggl[\frac{\lvert f'(a)\rvert^q+\lvert f'(b)\rvert^q}2\biggr]^{1/q}
\end{equation}
and
\begin{equation}
\biggl\lvert f\biggl(\frac{a+b}2\biggr)-\frac1{b-a}\int_a^bf(x)\td x\biggr\rvert\le\frac{b-a}4 \biggl[\frac{\lvert f'(a)\rvert^q+\lvert f'(b)\rvert^q}2\biggr]^{1/q}.
\end{equation}
\end{thm}

In~\cite{Kirmaci-AMC-04-146}, the above inequalities were further generalized as follows.

\begin{thm}[{\cite[Theorems~2.3 and 2.4]{Kirmaci-AMC-04-146}}]
Let $f:I\subseteq\mathbb{R}\to\mathbb{R}$ be differentiable on $I^\circ$, $a,b\in I$ with $a<b$, and $p>1$. If $\lvert f'(x)\rvert^{p/(p-1)}$ is convex on $[a,b]$, then
\begin{multline}
\biggl\lvert\frac1{b-a}\int_a^bf(x)\td x -f\biggl(\frac{a+b}2\biggr)\biggr\rvert \le\frac{b-a}{16}\biggl(\frac4{p+1}\biggr)^{1/p} \\*
\times\Bigl\{\bigl[\lvert f'(a)\rvert^{p/(p-1)}+3\lvert f'(b)\rvert^{p/(p-1)}\bigr]^{(p-1)/p}+\bigl[3\lvert f'(a)\rvert^{p/(p-1)}+\lvert f'(b)\rvert^{p/(p-1)}\bigr]^{(p-1)/p}\Bigr\}
\end{multline}
and
\begin{equation}
\biggl\lvert\frac1{b-a}\int_a^bf(x)\td x -f\biggl(\frac{a+b}2\biggr)\biggr\rvert \le\frac{b-a}4\biggl(\frac4{p+1}\biggr)^{1/p}(\lvert f'(a)\rvert +\lvert f'(b)\rvert ).
\end{equation}
\end{thm}

In~\cite{Dragomir-Agarwal-Cerone-JIA-579}, an inequality similar to the above ones was given as follows.

\begin{thm}[{\cite[Theorem~3]{Dragomir-Agarwal-Cerone-JIA-579}}]
Let $f:[a,b]\to\mathbb{R}$ be an absolutely continuous mapping on $[a,b]$ whose derivative belongs to $L_p[a,b]$. Then
\begin{equation}
\biggl\lvert\frac13\biggl[\frac{f(a)+f(b)}2+2f\biggl(\frac{a+b}2\biggr)\biggr]-\frac1{b-a}\int_a^bf(x)\td x \biggr\rvert
\le\frac16\biggl[\frac{2^{q+1}+1}{3(q+1)}\biggr]^{1/q}(b-a)^{1/q}\|f'\|_p,
\end{equation}
where $\frac1p+\frac1q=1$ and $p>1$.
\end{thm}

Recently, the following inequalities were obtained in~\cite{{Sarikaya-Aktan-1005.2897}, Sarikaya-Set-Ozdemir-rgmia-2010}.

\begin{thm}[\cite{Sarikaya-Aktan-1005.2897}]
Let $I\subseteq\mathbb{R}$ be an open interval, with $a,b\in I$ and $a<b$, and let $f:I\to\mathbb{R}$ be twice diferentiable mapping such that $f''(x)$ is integrable. If $0\le\lambda\le1$ and $|f''(x)|$ is a convex function on $[a,b]$, then
\begin{multline}
\biggl|(\lambda-1)f\biggl(\frac{a+b}2\biggr)-\lambda\frac{f(a)+f(b)}2+\int_a^bf(x)\td x\biggr|\\
\le
\begin{cases}
\begin{aligned}
&\dfrac{(b-a)^2}{24}\biggl\{\biggl[\lambda^4+(1+\lambda)(1-\lambda)^3+\dfrac{5\lambda-3}4\biggr]|f''(a)|\\
&\hskip8em +\biggl[\lambda^4+(2-\lambda)\lambda^3+\dfrac{1-3\lambda}4\biggr]|f''(b)|\biggr\}, \quad0\le\lambda\le\dfrac12;
\end{aligned}\\
\dfrac{(b-a)^2}{48}(3\lambda-1)\bigl(|f''(a)|+|f''(b)|\bigr),\quad\dfrac12\le\lambda\le1.
\end{cases}
\end{multline}
\end{thm}

\begin{thm}[\cite{Sarikaya-Set-Ozdemir-rgmia-2010}]\label{Sarikaya-Set-Ozdemir-rgmia-2010-thm}
Let $f:I\subseteq\mathbb{R}\to\mathbb{R}$ be differentiable on $I^\circ$, $a,b\in I$ with $a<b$, and $f'\in L[a,b]$. If $|f'(x)|^q$ is convex for $q\ge1$ on $[a,b]$, then
\begin{multline}
\biggl|\frac16\biggl[{f(a)+f(b)}+4f\biggl(\frac{a+b}2\biggr)\biggr] -\frac1{b-a}\int_a^bf(x)\td x\biggr|\\*
\le\frac{b-a}{12}\biggl[\frac{2^{q+1}+1}{3(q+1)}\biggr]^{1/q} \biggl[\biggl(\frac{3|f'(a)|^q+|f'(b)|^q}4\biggr)^{1/q} +\biggl(\frac{|f'(a)|^q+3|f'(b)|^q}4\biggr)^{1/q}\biggr]
\end{multline}
and
\begin{multline}
\biggl|\frac16\biggl[{f(a)+f(b)}+4f\biggl(\frac{a+b}2\biggr)\biggr] -\frac1{b-a}\int_a^bf(x)\td x\biggr|\\*
\le\frac{5(b-a)}{72} \biggl[\biggl(\frac{61|f'(a)|^q+29|f'(b)|^q}{90}\biggr)^{1/q} +\biggl(\frac{29|f'(a)|^q+61|f'(b)|^q}{90}\biggr)^{1/q} \biggr].
\end{multline}
\end{thm}

A function $f:I\subseteq(0, \infty)\to(0,\infty)$ is said to be $s$-convex if the inequality
\begin{equation}
f(\alpha x+\beta y)\le \alpha^sf(x)+\beta^sf(y)
\end{equation}
holds for all $x,y\in I$, $\alpha,\beta\in[0,1]$ with $\alpha+\beta=1$, and some fixed $s\in(0,1]$.
\par
In~\cite{Alomari-Darus-Kirmaci-2011-1643}, some inequalities of Hermite-Hadamard type for $s$-convex functions were established as follows.

\begin{thm}[{\cite[Theorems~2.3 and~2.4]{Alomari-Darus-Kirmaci-2011-1643}}]
Let $f:I\subset[0,\infty)\to\mathbb{R}$ be a differentiable mapping on $I^\circ$ such that $f'(x)\in L[a,b]$, where $a,b\in I$ with $a<b$.
\begin{enumerate}
  \item
  If $|f'(x)|^{p/(p-1)}$ is $s$-convex on $[a,b]$ for $p>1$ and some fixed $s\in(0,1]$, then
  \begin{equation}
  \begin{split}
    \biggl|f\biggl(\frac{a+b}2\biggr)-\frac1{b-a}\int_a^bf(x)\td x\biggr|
    &\le\frac{b-a}4\biggl(\frac1{p+1}\biggr)^{1/p}\biggl(\frac1{s+1}\biggr)^{2/q}\\
    &\quad\times\Bigl\{\bigl[\bigl(2^{1-s}+s+1\bigr)|f'(a)|^q+2^{1-s}|f'(b)|^q\bigr]^{1/q} \\
    &\quad+\bigl[2^{1-s}|f'(a)|^q+\bigl(2^{1-s}+s+1\bigr)|f'(b)|^q\bigr]^{1/q}\Bigr\},
    \end{split}
  \end{equation}
where $p$ is the conjugate of $q$, that is, $\frac1p+\frac1q=1$.
\item
If $|f'|^q$ is $s$-convex on $[a,b]$ for $q\ge1$ and some fixed $s\in(0,1]$, then
  \begin{multline}
    \biggl|f\biggl(\frac{a+b}2\biggr)-\frac1{b-a}\int_a^bf(x)\td x\biggr|
    \le\frac{b-a}8\biggl[\frac2{(s+1)(s+2)}\biggr]^{1/q}\\
    \times\Bigl\{\bigl[\bigl(2^{1-s}+1\bigr)|f'(a)|^q+2^{1-s}|f'(b)|^q\bigr]^{1/q} +\bigl[\bigl(2^{1-s}+1\bigr)|f'(b)|^q+2^{1-s}|f'(a)|^q\bigr]^{1/q}\Bigr\}.
  \end{multline}
\end{enumerate}
\end{thm}

In this paper we will establish some new Hermite-Hadamard type integral inequalities for functions whose first derivatives are of convexity and apply them to derive some inequalities of special means.

\section{Lemmas}

For establishing our new integral inequalities of Hermite-Hadamard type, we need the following lemmas.

\begin{lem}\label{lem1-September-2011-Qi-xi}
Let $I$ be an interval and $f:I\to\mathbb{R}$ be differentiable on $I^\circ$, with $a,b\in I$ and $a<b$, and $\lambda,\mu\in\mathbb{R}$. If $f'\in L[a,b]$, then
\begin{multline}\label{identity-half}
(1-\mu)f(a)+\lambda f(b)+(\mu-\lambda)f\biggl(\frac{a+b}2\biggr) -\frac1{b-a}\int_a^bf(x)\td x\\
=(b-a)\biggl[\int_0^{1/2}(\lambda-t)f'(ta+(1-t)b)\td t +\int_{1/2}^1(\mu-t)f'(ta+(1-t)b)\td t\biggr].
\end{multline}
\end{lem}

\begin{proof}
Integrating by part and changing variable of definite integral give
\begin{multline*}
\int_0^{1/2}(\lambda-t)f'(ta+(1-t)b)\td t +\int_{1/2}^1(\mu-t)f'(ta+(1-t)b)\td t\\
\begin{aligned}
&=\frac1{b-a}\biggl[(t-\lambda)f(ta+(1-t)b)\big|_{t=0}^{t=1/2} -\int_0^{1/2}f(ta+(1-t)b)\td t\\
&\quad+(t-\mu)f(ta+(1-t)b)\big|_{t=1/2}^{t=1} -\int_{1/2}^1f(ta+(1-t)b)\td t\biggr]\\
&=\frac1{b-a}\biggl[\biggl(\frac12-\lambda\biggr)f\biggl(\frac{a+b}2\biggr)+\lambda f(b)+\frac1{b-a}\int_b^{(a+b)/2}f(x)\td x\\
&\quad+(1-\mu)f(a)+\biggl(\mu-\frac12\biggr)f\biggl(\frac{a+b}2\biggr) +\frac1{b-a}\int_{(a+b)/2}^af(x)\td x\biggr]\\
&=\frac1{b-a}\biggl[(1-\mu)f(a)+\lambda f(b)+(\mu-\lambda)f\biggl(\frac{a+b}2\biggr) -\frac1{b-a}\int_a^bf(x)\td x\biggr].
\end{aligned}
\end{multline*}
The proof is complete.
\end{proof}

From Lemma~\ref{lem1-September-2011-Qi-xi} we can derive the following identity.

\begin{lem}[\cite{Hadramard-Convex-Xi-September-2011.tex}]\label{lem1-September-2011-xi}
Let $f:I\subseteq\mathbb{R}\to\mathbb{R}$ be differentiable on $I^\circ$, $a,b\in I$ with $a<b$. If $f'\in L[a,b]$ and $\lambda,\mu\in\mathbb{R}$, then
\begin{multline}\label{lem1-September-2011-xi-eq}
\frac{\lambda f(a)+\mu f(b)}2 +\frac{2-\lambda-\mu}2f\biggl(\frac{a+b}2\biggr) -\frac1{b-a}\int_a^bf(x)\td x\\*
=\frac{b-a}4\int_0^1\biggl[(1-\lambda-t)f'\biggl(ta+(1-t)\frac{a+b}2\biggr) +(\mu-t)f'\biggl(t\frac{a+b}2+(1-t)b\biggr)\biggr]\td t.
\end{multline}
\end{lem}

\begin{proof}
Replacing $\lambda$ and $\mu$ respectively by $\frac{\alpha}2$ and $1-\frac\beta2$ in the identity~\eqref{identity-half} yields
\begin{multline}\label{transform-vari}
\frac1{b-a}\biggl[\frac{\beta f(a)+{\alpha}f(b)}2+\frac{2-\alpha-\beta}2f\biggl(\frac{a+b}2\biggr) -\frac1{b-a}\int_a^bf(x)\td x\biggr]\\*
=\int_0^{1/2}\biggl(\frac{\alpha}2-t\biggr)f'(ta+(1-t)b)\td t +\int_{1/2}^1\biggl(1-\frac{\beta}2-t\biggr)f'(ta+(1-t)b)\td t.
\end{multline}
Changing variables of definite integral results in
\begin{equation}
\begin{split}\label{transform-vari-1}
\int_0^{1/2}\biggl(\frac{\alpha}2-t\biggr)f'(ta+(1-t)b)\td t
&=\frac14\int_0^1(\alpha-u)f'\biggl(\frac{u}2a+\frac{2-u}2b\biggr)\td u\\
&=\frac14\int_0^1(\alpha-u)f'\biggl(u\frac{a+b}2+(1-u)b\biggr)\td u
\end{split}
\end{equation}
and
\begin{equation}
\begin{split}\label{transform-vari-2}
\int_{1/2}^1\biggl(1-\frac{\beta}2-t\biggr)f'(ta+(1-t)b)\td t
&=\frac14\int_0^1(1-\beta-u)f'\biggl(\frac{1+u}2a+\frac{1-u}2b\biggr)\td u\\
&=\frac14\int_0^1(1-\beta-u)f'\biggl(ua+(1-u)\frac{a+b}2\biggr)\td u.
\end{split}
\end{equation}
Substituting \eqref{transform-vari-1} and~\eqref{transform-vari-2} into~\eqref{transform-vari} leads to
\begin{multline*}
\frac{\beta f(a)+{\alpha}f(b)}2+\frac{2-\alpha-\beta}2f\biggl(\frac{a+b}2\biggr) -\frac1{b-a}\int_a^bf(x)\td x\\
=\frac{b-a}4\int_0^1\biggl[(\alpha-u)f'\biggl(u\frac{a+b}2+(1-u)b\biggr) +(1-\beta-u)f'\biggl(ua+(1-u)\frac{a+b}2\biggr)\biggr]\td u,
\end{multline*}
which is equivalent to~\eqref{lem1-September-2011-xi-eq}. Lemma~\ref{lem1-September-2011-xi} is proved.
\end{proof}

\begin{rem}
The proof of Lemma~\ref{lem1-September-2011-xi} tells us that Lemma~\ref{lem1-September-2011-Qi-xi} and Lemma~\ref{lem1-September-2011-xi} are equivalent to each other.
\end{rem}

By taking $\lambda=\frac{\ell}m$ and $\mu=\frac{m-\ell}m$ for $m\neq0$ in Lemma~\ref{lem1-September-2011-Qi-xi}, we have the following identities.

\begin{lem}
Let $I$ be an interval and $f:I\to\mathbb{R}$ be differentiable on $I^\circ$, with $a,b\in I$ and $a<b$, and $m,\ell\in\mathbb{R}$ with $m\neq0$. If $f'\in L[a,b]$, then
\begin{multline}
\frac\ell{m}[f(a)+f(b)]+\frac{m-2\ell}mf\biggl(\frac{a+b}2\biggr)-\frac1{b-a}\int_a^bf(x)\td x\\*
   =(b-a)\biggl[\int_0^{1/2}\biggl(\frac\ell{m}-t\biggr)f'(ta+(1-t)b)\td t  +\int_{1/2}^1\biggl(\frac{m-\ell}m-t\biggr)f'(ta+(1-t)b)\td t\biggr].
\end{multline}
In particular, we have
\begin{multline}
f\biggl(\frac{a+b}2\biggr)-\frac1{b-a}\int_a^bf(x)\td x =(b-a)\biggl[\int_{1/2}^1(1-t)f'(ta+(1-t)b)\td t\\
-\int_0^{1/2}tf'(ta+(1-t)b)\td t\biggr],
\end{multline}
\begin{equation}
\frac{f(a)+f(b)}2-\frac1{b-a}\int_a^bf(x)\td x=\frac{b-a}2\int_0^1(1-2t)f'(ta+(1-t)b)\td t,
\end{equation}
\begin{multline}
\frac13\biggl[f(a)+f(b)+f\biggl(\frac{a+b}2\biggr)\biggr]-\frac1{b-a}\int_a^bf(x)\td x\\*
=(b-a)\biggl[\int_0^{1/2}\biggl(\frac13-t\biggr)f'(ta+(1-t)b)\td t +\int_{1/2}^1\biggl(\frac23-t\biggr)f'(ta+(1-t)b)\td t\biggr],
\end{multline}
\begin{multline}
\frac12\biggl[\frac{f(a)+f(b)}2+f\biggl(\frac{a+b}2\biggr)\biggr]-\frac1{b-a}\int_a^bf(x)\td x\\
=(b-a)\biggl[\int_0^{1/2}\biggl(\frac14-t\biggr)f'(ta+(1-t)b)\td t +\int_{1/2}^1\biggl(\frac34-t\biggr)f'(ta+(1-t)b)\td t\biggr],
\end{multline}
\begin{multline}
\frac15\biggl[f(a)+f(b)+3f\biggl(\frac{a+b}2\biggr)\biggr]-\frac1{b-a}\int_a^bf(x)\td x\\
=(b-a)\biggl[\int_0^{1/2}\biggl(\frac15-t\biggr)f'(ta+(1-t)b)\td t +\int_{1/2}^1\biggl(\frac45-t\biggr)f'(ta+(1-t)b)\td t\biggr],
\end{multline}
\begin{multline}
\frac15\biggl\{2[f(a)+f(b)]+f\biggl(\frac{a+b}2\biggr)\biggr\}-\frac1{b-a}\int_a^bf(x)\td x\\
=(b-a)\biggl[\int_0^{1/2}\biggl(\frac25-t\biggr)f'(ta+(1-t)b)\td t +\int_{1/2}^1\biggl(\frac35-t\biggr)f'(ta+(1-t)b)\td t\biggr],
\end{multline}
\begin{multline}
\frac16\biggl[f(a)+f(b)+4f\biggl(\frac{a+b}2\biggr)\biggr]-\frac1{b-a}\int_a^bf(x)\td x\\*
=(b-a)\biggl[\int_0^{1/2}\biggl(\frac16-t\biggr)f'(ta+(1-t)b)\td t +\int_{1/2}^1\biggl(\frac56-t\biggr)f'(ta+(1-t)b)\td t\biggr].
\end{multline}
\end{lem}

\section{New integral inequalities of Hermite-Hadamard type}

Now we are in a position to establish some new integral inequalities of Hermite-Hadamard type for functions whose derivatives are of convexity.

\begin{thm}\label{Qi-Xi-Sep-2011-main-thm-1}
Let $f:I\subseteq\mathbb{R}\to\mathbb{R}$ be a differentiable function on $I^\circ$, $a,b\in I$ with $a<b$, $0\le\lambda\le\frac12\le\mu\le1$, and $f'\in L[a,b]$. If $\lvert f'(x)\rvert$ is convex on $[a,b]$, then
\begin{multline}\label{Qi-Xi-Sep-2011-main-ineq}
\biggl|(1-\mu)f(a)+\lambda f(b)+(\mu-\lambda)f\biggl(\frac{a+b}2\biggr) -\frac1{b-a}\int_a^bf(x)\td x\biggr|\\
\le\frac{b-a}{24}\bigl[\bigl(10-3\lambda+8\lambda^3-15\mu+8\mu^3\bigr)|f'(a)|\\
+\bigl(8-9\lambda+24\lambda^2-8\lambda^3-21\mu+24\mu^2-8\mu^3\bigr)|f'(b)|\bigr].
\end{multline}
\end{thm}

\begin{proof}
By Lemma~\ref{lem1-September-2011-Qi-xi} and the convexity of $\lvert f'(x)\rvert$ on $[a,b]$, we have
\begin{multline*}
\biggl|(1-\mu)f(a)+\lambda f(b)+(\mu-\lambda)f\biggl(\frac{a+b}2\biggr) -\frac1{b-a}\int_a^bf(x)\td x\biggr|\\
\le(b-a)\biggl[\int_0^{1/2}|\lambda-t||f'(ta+(1-t)b)|\td t +\int_{1/2}^1|\mu-t||f'(ta+(1-t)b)|\td t\biggr]\\
\le(b-a)\biggl[\int_0^{1/2}|\lambda-t|\bigl(t|f'(a)|+(1-t)|f'(b)|\bigr)\td t+ \int_{1/2}^1|\mu-t|\bigl(t|f'(a)|+(1-t)|f'(b)|\bigr)\td t\biggr].
\end{multline*}
Substituting equations
\begin{multline*}
\int_0^{1/2}|\lambda-t|\bigl(t|f'(a)|+(1-t)|f'(b)|\bigr)\td t \\*
=\frac1{24}\bigl[\bigl(1-3\lambda+8\lambda^3\bigr)|f'(a)| +\bigl(2-9\lambda+24\lambda^2-8\lambda^3\bigr)|f'(b)|\bigr]
\end{multline*}
and
\begin{multline*}
\int_{1/2}^1|\mu-t|\bigl(t|f'(a)|+(1-t)|f'(b)|\bigr)\td t \\
=\frac1{24}\bigl[\bigl(9-15\mu+8\mu^3\bigr)|f'(a)| +\bigl(6-21\mu+24\mu^2-8\mu^3\bigr)|f'(b)|\bigr]
\end{multline*}
into the above inequality leads to \eqref{Qi-Xi-Sep-2011-main-ineq}.
The proof of Theorem~\ref{Qi-Xi-Sep-2011-main-thm-1} is complete.
\end{proof}

\begin{thm}\label{Qi-Xi-Sep-2011-main-thm}
Let $f:I\subseteq\mathbb{R}\to\mathbb{R}$ be a differentiable function on $I^\circ$, $a,b\in I$ with $a<b$, $0\le\lambda\le\frac12\le\mu\le1$, and $f'\in L[a,b]$. If $\lvert f'(x)\rvert^q$ for $q>1$ is convex on $[a,b]$ and $q\ge p>0$, then
\begin{multline}
\biggl|(1-\mu)f(a)+\lambda f(b)+(\mu-\lambda)f\biggl(\frac{a+b}2\biggr) -\frac1{b-a}\int_a^bf(x)\td x\biggr|\\*
\begin{aligned}
&\le(b-a)\biggl(\frac{q-1}{2q-p-1}\biggr)^{1-1/q}\biggl[\frac1{(p+1)(p+2)}\biggr]^{1/q}\\
&\quad\times\biggl\{\biggl[\biggl(\frac12-\lambda\biggr)^{(2q-p-1)/(q-1)} +\lambda^{(2q-p-1)/(q-1)}\biggr]^{1-1/q}\\
&\quad\times\biggl(\biggl[\frac12(p+1+2\lambda)\biggl(\frac12-\lambda\biggr)^{p+1} +\lambda^{p+2}\biggr]|f'(a)|^q\\
&\quad+\biggl[\frac12(p+3-2\lambda)\biggl(\frac12-\lambda\biggr)^{p+1} +(p+2-\lambda)\lambda^{p+1}\biggr]|f'(b)|^q\biggr)^{1/q}\label{Qi-Xi-Sep-2011-main-ineq-3.2}\\
&\quad+\biggl[\biggl(\mu-\frac12\biggr)^{(2q-p-1)/(q-1)}+(1-\mu)^{(2q-p-1)/(q-1)}\biggr]^{1-1/q}\\
&\quad\times\biggl(\biggl[\frac12(p+1+2\mu)\biggl(\mu-\frac12\biggr)^{p+1} +(p+1+\mu)(1-\mu)^{p+1}\biggr]|f'(a)|^q\\
&\quad+\biggl[\frac12(p+3-2\mu)\biggl(\mu-\frac12\biggr)^{p+1} +(1-\mu)^{p+2}\biggr]|f'(b)|^q\biggr)^{1/q}\biggr\}.
\end{aligned}
\end{multline}
\end{thm}

\begin{proof}
By Lemma~\ref{lem1-September-2011-Qi-xi}, the convexity of $\lvert f'(x)\rvert^q$ on $[a,b]$, and H\"older's integral inequality, we have
\begin{multline}\label{split-eq-Qi-Xi-main}
\biggl|(1-\mu)f(a)+\lambda f(b)+(\mu-\lambda)f\biggl(\frac{a+b}2\biggr) -\frac1{b-a}\int_a^bf(x)\td x\biggr|\\*
\begin{aligned}
&\le(b-a)\biggl[\int_0^{1/2}|\lambda-t||f'(ta+(1-t)b)|\td t +\int_{1/2}^1|\mu-t||f'(ta+(1-t)b)|\td t\biggr]\\
&\le(b-a)\biggl[\biggl(\int_0^{1/2}|\lambda-t|^{(q-p)/(q-1)}\td t\biggr)^{1-1/q} \biggl(\int_0^{1/2}|\lambda-t|^p|f'(ta+(1-t)b)|^q\td t\biggr)^{1/q}\\
&\quad+\biggl(\int_{1/2}^1|\mu-t|^{(q-p)/(q-1)}\td t\biggr)^{1-1/q} \biggl(\int_{1/2}^1|\mu-t|^p|f'(ta+(1-t)b)|^q\td t\biggr)^{1/q}\biggr]\\
&\le(b-a)\biggl\{\biggl[\int_0^{1/2}|\lambda-t|^{(q-p)/(q-1)}\td t\biggr]^{1-1/q} \biggl[\int_0^{1/2}|\lambda-t|^p\bigl(t|f'(a)|^q+(1-t)|f'(b)|^q\bigr)\td t\biggr]^{1/q}\\
&\quad+\biggl[\int_{1/2}^1|\mu-t|^{(q-p)/(q-1)}\td t\biggr]^{1-1/q} \biggl[\int_{1/2}^1|\mu-t|^p\bigl(t|f'(a)|^q+(1-t)|f'(b)|^q\bigr)\td t\biggr]^{1/q}\biggr\}.
\end{aligned}
\end{multline}
Furthermore, a straightforward computation gives
\begin{gather}\label{equat-3.3}
\int_0^{1/2}|\lambda-t|^{(q-p)/(q-1)}\td t=\frac{q-1}{2q-p-1} \biggl[\biggl(\frac12-\lambda\biggr)^{(2q-p-1)/(q-1)}+\lambda^{(2q-p-1)/(q-1)}\biggr],\\
\int_{1/2}^1|\mu-t|^{(q-p)/(q-1)}\td t=\frac{q-1}{2q-p-1} \biggl[\biggl(\mu-\frac12\biggr)^{(2q-p-1)/(q-1)}+(1-\mu)^{(2q-p-1)/(q-1)}\biggr],
\end{gather}
\begin{multline}
\int_0^{1/2}|\lambda-t|^p\bigl(t|f'(a)|^q+(1-t)|f'(b)|^q\bigr)\td t=\frac1{(p+1)(p+2)}\\*
\times\biggl\{\biggl[\frac12(p+1+2\lambda) \biggl(\frac12-\lambda\biggr)^{p+1}+\lambda^{p+2}\biggr]|f'(a)|^q\\
+\biggl[\frac12(p+3-2\lambda) \biggl(\frac12-\lambda\biggr)^{p+1} +(p+2-\lambda)\lambda^{p+1}\biggr]|f'(b)|^q\biggr\},
\end{multline}
\begin{multline}\label{equat-3.6}
\int_{1/2}^1|\mu-t|^p\bigl(t|f'(a)|^q+(1-t)|f'(b)|^q\bigr)\td t=\frac1{(p+1)(p+2)}\\
\times\biggl\{\biggl[\frac12(p+1+2\mu) \biggl(\mu-\frac12\biggr)^{p+1}+(p+1+\mu)(1-\mu)^{p+1}\biggr]|f'(a)|^q\\
+\biggl[\frac12(p+3-2\mu) \biggl(\mu-\frac12\biggr)^{p+1} +(1-\mu)^{p+2}\biggr]|f'(b)|^q\biggr\}.
\end{multline}
Substituting the equations from \eqref{equat-3.3} to \eqref{equat-3.6} into~\eqref{split-eq-Qi-Xi-main} results in the inequality~\eqref{Qi-Xi-Sep-2011-main-ineq}.
The proof of Theorem~\ref{Qi-Xi-Sep-2011-main-thm} is complete.
\end{proof}

\begin{cor}\label{Qi-Xi-Sep-2011-Cor-3.1.1}
Let $f:I\subseteq\mathbb{R}\to\mathbb{R}$ be a differentiable function on $I^\circ$, $a,b\in I$ with $a<b$, $0\le\lambda\le\frac12\le\mu\le1$, and $f'\in L[a,b]$. If $\lvert f'(x)\rvert^q$ for $q\ge 1$ is convex on $[a,b]$, then
\begin{multline}\label{Qi-Xi-Sep-2011-cor1-ineq-3.8}
\biggl|(1-\mu)f(a)+\lambda f(b)+(\mu-\lambda)f\biggl(\frac{a+b}2\biggr) -\frac1{b-a}\int_a^bf(x)\td x\biggr|\\
\begin{aligned}
&\le\frac{b-a}2\biggl(\frac1{3}\biggr)^{1/q}\biggl\{\biggl[\biggl(\frac12-\lambda\biggr)^2 +\lambda^2\biggr]^{1-1/q}\\
&\quad\times\biggl(\biggl[(1+\lambda)\biggl(\frac12-\lambda\biggr)^2 +\lambda^{3}\biggr]|f'(a)|^q+\biggl[(2-\lambda)\biggl(\frac12-\lambda\biggr)^2 +(3-\lambda)\lambda^2\biggr]|f'(b)|^q\biggr)^{1/q}\\
&\quad+\biggl[\biggl(\mu-\frac12\biggr)^2+(1-\mu)^2\biggr]^{1-1/q}
\biggl(\biggl[(1+\mu)\biggl(\mu-\frac12\biggr)^2 +(2+\mu)(1-\mu)^2\biggr]|f'(a)|^q\\
&\quad+\biggl[(2-\mu)\biggl(\mu-\frac12\biggr)^2 +(1-\mu)^{3}\biggr]|f'(b)|^q\biggr)^{1/q}\biggr\}
\end{aligned}
\end{multline}
and
\begin{multline}\label{Qi-Xi-Sep-2011-cor1-ineq-3.9}
\biggl|(1-\mu)f(a)+\lambda f(b)+(\mu-\lambda)f\biggl(\frac{a+b}2\biggr) -\frac1{b-a}\int_a^bf(x)\td x\biggr|\\*
\begin{aligned}
&\le\frac{b-a}2\biggl[\frac2{(q+1)(q+2)}\biggr]^{1/q}
\biggl\{\biggl(\biggl[\frac12(q+1+2\lambda)\biggl(\frac12-\lambda\biggr)^{q+1} +\lambda^{q+2}\biggr]|f'(a)|^q\\
&\quad+\biggl[\frac12(q+3-2\lambda)\biggl(\frac12-\lambda\biggr)^{q+1} +(q+2-\lambda)\lambda^{q+1}\biggr]|f'(b)|^q\biggr)^{1/q}\\
&\quad+\biggl(\biggl[\frac12(q+1+2\mu)\biggl(\mu-\frac12\biggr)^{q+1} +(q+1+\mu)(1-\mu)^{q+1}\biggr]|f'(a)|^q\\
&\quad+\biggl[\frac12(q+3-2\mu)\biggl(\mu-\frac12\biggr)^{q+1} +(1-\mu)^{q+2}\biggr]|f'(b)|^q\biggr)^{1/q}\biggr\}.
\end{aligned}
\end{multline}
\end{cor}

\begin{proof}
This follows from Theorem~\ref{Qi-Xi-Sep-2011-main-thm-1} and setting $p=1$ and $p=q$ in  Theorem~\ref{Qi-Xi-Sep-2011-main-thm}.
\end{proof}

\begin{cor}\label{Qi-Xi-Sep-2011-Cor-3.1.2}
Let $f:I\subseteq\mathbb{R}\to\mathbb{R}$ be a differentiable function on $I^\circ$, $a,b\in I$ with $a<b$, $m>0$ and $m\ge2\ell\ge0$, and $f'\in L[a,b]$. If $\lvert f'(x)\rvert^q$ for $q>1$ is convex on $[a,b]$, and $q\ge p>0$, then
\begin{multline}\label{Qi-Xi-Sep-2011-cor1-ineq-3.10}
\biggl|\frac{\ell}m[f(a)+f(b)]+\frac{m-2\ell}mf\biggl(\frac{a+b}2\biggr) -\frac1{b-a}\int_a^bf(x)\td x\biggr|\\*
\begin{aligned}
&\le\frac{b-a}{4m^2}\biggr(\frac{q-1}{2q-p-1}\biggr)^{1-1/q}\biggr(\frac1{2m(p+1)(p+2)}\biggr)^{1/q}\\
&\quad\times[(2\ell)^{(2q-p-1)/(q-1)}+(m-2\ell)^{(2q-p-1)/(q-1)}]^{1-1/q}\\
&\quad\times\Bigl\{\bigl([(2\ell)^{p+2}+(mp+m+2\ell)(m-2\ell)^{p+1}]|f'(a)|^q\\
&\quad+[(2mp+4m-2\ell)(2\ell)^{p+1}+(mp+3m-2\ell)(m-2\ell)^{p+1}]|f'(b)|^q\bigr)^{1/q}\\
&\quad+\bigl([(2mp+4m-2\ell)(2\ell)^{p+1} +(mp+3m-2\ell)(m-2\ell)^{p+1}]|f'(a)|^q\\
&\quad+[(2\ell)^{p+2}+(mp+m+2\ell)(m-2\ell)^{p+1}]|f'(b)|^q\bigr)^{1/q}\Bigr\}.
\end{aligned}
\end{multline}
\end{cor}
\begin{proof}
  This follows from letting $\lambda=1-\mu=\frac{\ell}m$ in Theorem~\ref{Qi-Xi-Sep-2011-main-thm}.
\end{proof}

\begin{cor}\label{Qi-Xi-Sep-2011-Cor-3.1.3}
Let $f:I\subseteq\mathbb{R}\to\mathbb{R}$ be a differentiable function on $I^\circ$, $a,b\in I$ with $a<b$, $m>0$ and $m\ge2\ell\ge0$, and $f'\in L[a,b]$. If $\lvert f'(x)\rvert^q$ for $q\ge1$ is convex on $[a,b]$, then
\begin{multline}\label{Qi-Xi-Sep-2011-cor1-ineq-3.12--p=1}
\biggl|\frac{\ell}m[f(a)+f(b)]+\frac{m-2\ell}mf\biggl(\frac{a+b}2\biggr) -\frac1{b-a}\int_a^bf(x)\td x\biggr|\\*
\begin{aligned}
&\le\frac{b-a}{8m^2}\biggr(\frac{1}{3m}\biggr)^{1/q}[(m-2\ell)^2+(2\ell)^2]^{1-1/q}
\Bigl\{\bigl([(m+\ell)(m-2\ell)^2 +4\ell^{3}]|f'(a)|^q\\
&\quad+[(2m-\ell)(m-2\ell)^2 +4(3m-\ell)\ell^2]|f'(b)|^q\bigr)^{1/q}+\bigl([4(3m-\ell)\ell^2\\
&\quad +(2m-\ell)(m-2\ell)^2]|f'(a)|^q +[(m+\ell)(m-2\ell)^2 +4\ell^{3}]|f'(b)|^q\bigr)^{1/q}\Bigr\}
\end{aligned}
\end{multline}
and
\begin{multline}\label{Qi-Xi-Sep-2011-cor1-ineq-3.11--p=q}
\biggl|\frac{\ell}m[f(a)+f(b)]+\frac{m-2\ell}mf\biggl(\frac{a+b}2\biggr) -\frac1{b-a}\int_a^bf(x)\td x\biggr|\\
\begin{aligned}
&\le\frac{b-a}{4m}\biggr[\frac1{2m^2(q+1)(q+2)}\biggr]^{1/q}\Bigl\{\bigl([(2\ell)^{q+2}
+(mq+m+2\ell)(m-2\ell)^{q+1} ]|f'(a)|^q\\
&\quad+[(2mq+4m-2\ell)(2\ell)^{q+1} +(mq+3m-2\ell)(m-2\ell)^{q+1}]|f'(b)|^q\bigr)^{1/q}\\
&\quad+\bigl([(2mq+4m-2\ell)(2\ell)^{q+1} +(mq+3m-2\ell)(m-2\ell)^{q+1}]|f'(a)|^q\\
&\quad+[(2\ell)^{q+2}+(mq+m+2\ell)(m-2\ell)^{q+1}]|f'(b)|^q\bigr)^{1/q}\Bigr\}.
\end{aligned}
\end{multline}
\end{cor}

\begin{proof}
This follows from setting $\lambda=1-\mu=\frac{\ell}m$ in Corollary~\ref{Qi-Xi-Sep-2011-Cor-3.1.1}.
\end{proof}

\begin{cor}\label{Qi-Xi-Sep-2011-Cor-3.1.4}
Let $f:I\subseteq\mathbb{R}\to\mathbb{R}$ be a differentiable function on $I^\circ$, $a,b\in I$ with $a<b$,  $q\ge p>0$, and $f'\in L[a,b]$. If $\lvert f'(x)\rvert^q$ for $q>1$ is convex on $[a,b]$, then
\begin{multline}\label{Qi-Xi-Sep-2011-cor1-ineq-3.13}
\biggl|f\biggl(\frac{a+b}2\biggr) -\frac1{b-a}\int_a^bf(x)\td x\biggr|
\le\frac{b-a}{4}\biggr(\frac{q-1}{2q-p-1}\biggr)^{1-1/q}\biggr[\frac1{2(p+1)(p+2)}\biggr]^{1/q}\\
\times\Bigl\{\bigr[{(p+1)|f'(a)|^q+(p+3)|f'(b)|^q}\bigr]^{1/q}
+\bigr[{(p+3)|f'(a)|^q+(p+1)|f'(b)|^q}\bigr]^{1/q}\Bigr\},
\end{multline}
\begin{multline}\label{Qi-Xi-Sep-2011-cor1-ineq-3.14}
\biggl|\frac{f(a)+f(b)}2 -\frac1{b-a}\int_a^bf(x)\td x\biggr|
\le\frac{b-a}{4}\biggr(\frac{q-1}{2q-p-1}\biggr)^{1-1/q}\biggr[\frac1{2(p+1)(p+2)}\biggr]^{1/q}\\
\times\Bigl\{\bigr[{|f'(a)|^q+(2p+3)|f'(b)|^q}\bigr]^{1/q}
+\bigr[{(2p+3)|f'(a)|^q+|f'(b)|^q}\bigr]^{1/q}\Bigr\},
\end{multline}
\begin{multline}\label{Qi-Xi-Sep-2011-cor1-ineq-3.15}
\biggl|\frac{1}3\biggr[f(a)+f(b)+f\biggl(\frac{a+b}2\biggr)\biggr] -\frac1{b-a}\int_a^bf(x)\td x\biggr|\\
\le\frac{b-a}{36}\biggr(\frac{q-1}{2q-p-1}\biggr)^{1-1/q}\biggr[\frac1{6(p+1)(p+2)}\biggr]^{1/q}
\bigl[1+2^{(2q-p-1)/(q-1)}\bigr]^{1-1/q}\\
\times\Bigl\{\bigl[\bigl(2^{p+2}+3p+5\bigr)|f'(a)|^q+\bigl((3p+5)2^{p+2}+3p+7\bigr)|f'(b)|^q\bigr]^{1/q}\\
+\bigl[\bigl((3p+5)2^{p+2}+3p+7\bigr)|f'(a)|^q+\bigl(2^{p+2}+3p+5\bigr)|f'(b)|^q\bigr]^{1/q}\Bigr\},
\end{multline}
\begin{multline}\label{Qi-Xi-Sep-2011-cor1-ineq-3.16}
\biggl|\frac12\biggl[\frac{f(a)+f(b)}2+f\biggl(\frac{a+b}2\biggr)\biggr]-\frac1{b-a}\int_a^bf(x)\td x\biggr| \le\frac{b-a}{8}\biggr(\frac{q-1}{2q-p-1}\biggr)^{1-1/q}\\
\times\biggr[\frac1{4(p+1)(p+2)}\biggr]^{1/q} \Bigl\{\bigr[{(p+2)|f'(a)|^q+(3p+6)|f'(b)|^q}\bigr]^{1/q}\\
+\bigr[{(3p+6)|f'(a)|^q+(p+2)|f'(b)|^q}\bigr]^{1/q}\Bigr\},
\end{multline}
\begin{multline}\label{Qi-Xi-Sep-2011-cor1-ineq-3.17}
\biggl|\frac{1}5\biggr[f(a)+f(b)+3f\biggl(\frac{a+b}2\biggr)\biggr] -\frac1{b-a}\int_a^bf(x)\td x\biggr|\\
\le\frac{b-a}{100}\biggr(\frac{q-1}{2q-p-1}\biggr)^{1-1/q}\biggr[\frac1{10(p+1)(p+2)}\biggr]^{1/q}
\bigl[2^{(2q-p-1)/(q-1)}+3^{(2q-p-1)/(q-1)}\bigr]^{1-1/q}\\
\times\Bigl\{\bigl([2^{p+2}+(5p+7)3^{p+1}]|f'(a)|^q+[(5p+9)2^{p+2} +(5p+13)3^{p+1}]|f'(b)|^q\bigr)^{1/q}\\
+\bigl([(5p+9)2^{p+2} +(5p+13)3^{p+1}]|f'(a)|^q +[2^{p+2}+(5p+7)3^{p+1}]|f'(b)|^q\bigr)^{1/q}\Bigr\},
\end{multline}
\begin{multline}\label{Qi-Xi-Sep-2011-cor1-ineq-3.18}
\biggl|\frac{1}5\biggr\{2[f(a)+f(b)]+f\biggl(\frac{a+b}2\biggr)\biggr\} -\frac1{b-a}\int_a^bf(x)\td x\biggr|\\
\le\frac{b-a}{100}\biggr(\frac{q-1}{2q-p-1}\biggr)^{1-1/q} \biggr[\frac1{10(p+1)(p+2)}\biggr]^{1/q}\bigl[1+4^{(2q-p-1)/(q-1)}\bigr]^{1-1/q}\\
\times\Bigl\{\bigl((4^{p+2}+5p+9)|f'(a)|^q+[(5p+8)2^{2p+3}+5p+11]|f'(b)|^q\bigr)^{1/q}\\
+\bigl([(5p+8)2^{2p+3}+5p+11]|f'(a)|^q+(4^{p+2}+5p+9)|f'(b)|^q\bigr)^{1/q}\Bigr\},
\end{multline}
and
\begin{multline}\label{Qi-Xi-Sep-2011-cor1-ineq-3.19}
\biggl|\frac{1}6\biggr[f(a)+f(b)+4f\biggl(\frac{a+b}2\biggr)\biggr] -\frac1{b-a}\int_a^bf(x)\td x\biggr|\\
\le\frac{b-a}{36}\biggr(\frac{q-1}{2q-p-1}\biggr)^{1-1/q}\biggr[\frac1{6(p+1)(p+2)}\biggr]^{1/q}
\bigl[1+2^{(2q-p-1)/(q-1)}\bigr]^{1-1/q}\\
\times\Bigl\{\bigl([(3p+4)2^{p+1}+1]|f'(a)|^q+[(3p+8)2^{p+1}+6p+11]|f'(b)|^q\bigr)^{1/q}\\
+\bigl([(3p+8)2^{p+1}+6p+11]|f'(a)|^q+[(3p+4)2^{p+1}+1]|f'(b)|^q\bigr)^{1/q}\Bigr\}.
\end{multline}
\end{cor}

\begin{proof}
These follow from letting $(m, \ell)=(1,0)$, $(2,1)$, $(3,1)$, $(4,1)$, $(5,1)$, $(5,2)$, and $(6,1)$ in Corollary~\ref{Qi-Xi-Sep-2011-Cor-3.1.2}, respectively.
\end{proof}

\begin{cor}\label{Qi-Xi-Sep-2011-Cor-3.1.5}
Let $f:I\subseteq\mathbb{R}\to\mathbb{R}$ be a differentiable function on $I^\circ$, $a,b\in I$ with $a<b$, $q\ge p>0$, and $f'\in L[a,b]$. If $\lvert f'(x)\rvert^q$ for $q\ge1$ is convex on $[a,b]$, then
\begin{multline}\label{Qi-Xi-Sep-2011-cor1-ineq-3.20}
\biggl|f\biggl(\frac{a+b}2\biggr) -\frac1{b-a}\int_a^bf(x)\td x\biggr|
\le\frac{b-a}{4}\biggr[\frac1{2(q+1)(q+2)}\biggr]^{1/q}\\
\times\Bigl\{\bigr[{(q+1)|f'(a)|^q+(q+3)|f'(b)|^q}\bigr]^{1/q}
+\bigr[{(q+3)|f'(a)|^q+(q+1)|f'(b)|^q}\bigr]^{1/q}\Bigr\},
\end{multline}
\begin{multline}\label{Qi-Xi-Sep-2011-cor1-ineq-3.21}
\biggl|\frac{f(a)+f(b)}2 -\frac1{b-a}\int_a^bf(x)\td x\biggr|
\le\frac{b-a}{4}\biggr[\frac1{2(q+1)(q+2)}\biggr]^{1/q}\\
\times\Bigl\{\bigr[{|f'(a)|^q+(2q+3)|f'(b)|^q}\bigr]^{1/q}
+\bigr[{(2q+3)|f'(a)|^q+|f'(b)|^q}\bigr]^{1/q}\Bigr\},
\end{multline}
\begin{multline}\label{Qi-Xi-Sep-2011-cor1-ineq-3.22}
\biggl|\frac{1}3\biggr[f(a)+f(b)+f\biggl(\frac{a+b}2\biggr)\biggr] -\frac1{b-a}\int_a^bf(x)\td x\biggr|
\le\frac{b-a}{12}\biggr[\frac1{18(q+1)(q+2)}\biggr]^{1/q}\\
\times\Bigl\{\bigl((2^{q+2}+3q+5)|f'(a)|^q+[(3q+5)2^{q+2}+3q+7]|f'(b)|^q\bigr)^{1/q}\\
+\bigl([(3q+5)2^{q+2}+3q+7]|f'(a)|^q+[2^{q+2}+3q+5]|f'(b)|^q\bigr)^{1/q}\Bigr\},
\end{multline}
\begin{multline}\label{Qi-Xi-Sep-2011-cor1-ineq-3.23}
\biggl|\frac12\biggl[\frac{f(a)+f(b)}2 +f\bigg(\frac{a+b}2\biggr)\biggr]-\frac1{b-a}\int_a^bf(x)\td x\biggr|
\le\frac{b-a}{8}\biggr[\frac1{4(q+1)(q+2)}\biggr]^{1/q}\\
\times\Bigl\{\bigr[{(q+2)|f'(a)|^q+(3q+6)|f'(b)|^q}{4(q+2)}\bigr]^{1/q}\\
+\bigr[{(3q+6)|f'(a)|^q+(q+2)|f'(b)|^q}{4(q+2)}\bigr]^{1/q}\Bigl\},
\end{multline}
\begin{multline}\label{Qi-Xi-Sep-2011-cor1-ineq-3.24}
\biggl|\frac{1}5\biggr[f(a)+f(b)+3f\biggl(\frac{a+b}2\biggr)\biggr] -\frac1{b-a}\int_a^bf(x)\td x\biggr|
\le\frac{b-a}{20}\biggr[\frac1{50(q+1)(q+2)}\biggr]^{1/q}\\
\times\Bigl\{\bigl([2^{q+2}+(5q+7)3^{q+1}]|f'(a)|^q+[(5q+9)2^{q+2}+(5q+13)3^{q+1}]|f'(b)|^q\bigr)^{1/q}\\
+\bigl([(5q+9)2^{q+2} +(5q+13)3^{q+1}]|f'(a)|^q+[2^{q+2}+(5q+7)3^{q+1}]|f'(b)|^q\bigr)^{1/q}\Bigr\},
\end{multline}
\begin{multline}\label{Qi-Xi-Sep-2011-cor1-ineq-3.25}
\biggl|\frac{1}5\biggr[2[f(a)+f(b)]+f\biggl(\frac{a+b}2\biggr)\biggr] -\frac1{b-a}\int_a^bf(x)\td x\biggr|
\le\frac{b-a}{20}\biggr[\frac1{50(q+1)(q+2)}\biggr]^{1/q}\\
\times\Bigl\{\bigl([4^{q+2}+5q+9]|f'(a)|^q+[(5q+8)2^{2q+3}+5q+11]|f'(b)|^q\bigr)^{1/q}\\
+\bigl([(5q+8)2^{2q+3}+5q+11]|f'(a)|^q+[4^{q+2}+5q+9]|f'(b)|^q\bigr)^{1/q}\Bigr\},
\end{multline}
\begin{multline}\label{Qi-Xi-Sep-2011-cor1-ineq-3.26}
\biggl|\frac{1}6\biggr[f(a)+f(b)+4f\biggl(\frac{a+b}2\biggr)\biggr] -\frac1{b-a}\int_a^bf(x)\td x\biggr|
\le\frac{b-a}{12}\biggr[\frac1{18(p+1)(p+2)}\biggr]^{1/q} \\
\times\Bigl\{\bigl([(3q+4)2^{q+1}+1]|f'(a)|^q+[(3q+8)2^{q+1}+6q+11]|f'(b)|^q\bigr)^{1/q}\\
+\bigl([(3q+8)2^{q+1}+6q+11]|f'(a)|^q+[(3q+4)2^{q+1}+1]|f'(b)|^q\bigr)^{1/q}\Bigr\}.
\end{multline}
\end{cor}

\begin{proof}
These inequalities follow from letting $(m, \ell)=(1,0)$, $(2,1)$, $(3,1)$, $(4,1)$, $(5,1)$, $(5,2)$, and $(6,1)$ in ~\eqref{Qi-Xi-Sep-2011-cor1-ineq-3.11--p=q}, respectively.
\end{proof}

\begin{cor}\label{Qi-Xi-Sep-2011-Cor-3.1.6}
Let $f:I\subseteq\mathbb{R}\to\mathbb{R}$ be a differentiable function on $I^\circ$, $a,b\in I$ with $a<b$, and $f'\in L[a,b]$. If $\lvert f'(x)\rvert^q$ for $q\ge1$ is convex on $[a,b]$, then
\begin{multline}\label{Qi-Xi-Sep-2011-cor1-ineq-3.27-2}
\biggl|f\biggl(\frac{a+b}2\biggr) -\frac1{b-a}\int_a^bf(x)\td x\biggr|\\
\le\frac{b-a}{8}\biggr[\biggr(\frac{|f'(a)|^q+2|f'(b)|^q}{3}\biggr)^{1/q}
+\biggr(\frac{2|f'(a)|^q+|f'(b)|^q}{3}\biggr)^{1/q}\biggr],
\end{multline}
\begin{multline}\label{Qi-Xi-Sep-2011-cor1-ineq-3.28-2}
\biggl|\frac{f(a)+f(b)}2 -\frac1{b-a}\int_a^bf(x)\td x\biggr|\\
\le\frac{b-a}{8}\biggr[\biggr(\frac{|f'(a)|^q+5|f'(b)|^q}{6}\biggr)^{1/q}
+\biggr(\frac{5|f'(a)|^q+|f'(b)|^q}{6}\biggr)^{1/q}\biggr],
\end{multline}
\begin{multline}\label{Qi-Xi-Sep-2011-cor1-ineq-3.29-2}
\biggl|\frac{1}3\biggr[f(a)+f(b)+f\biggl(\frac{a+b}2\biggr)\biggr] -\frac1{b-a}\int_a^bf(x)\td x\biggr|\\
\le\frac{5(b-a)}{72}\biggr[\biggr(\frac{8|f'(a)|^q+37|f'(b)|^q}{45}\biggr)^{1/q}
+\biggr(\frac{37|f'(a)|^q+8|f'(b)|^q}{45}\biggr)^{1/q}\biggr],
\end{multline}
\begin{multline}\label{Qi-Xi-Sep-2011-cor1-ineq-3.30-2}
\biggl|\frac12\biggl[\frac{f(a)+f(b)}2 +f\biggl(\frac{a+b}2\biggr)\biggr] -\frac1{b-a}\int_a^bf(x)\td x\biggr|\\
\le\frac{b-a}{16}\biggr[\biggr(\frac{|f'(a)|^q+3|f'(b)|^q}{4}\biggr)^{1/q}
+\biggr(\frac{3|f'(a)|^q+|f'(b)|^q}{4}\biggr)^{1/q}\biggr],
\end{multline}
\begin{multline}\label{Qi-Xi-Sep-2011-cor1-ineq-3.31-2}
\biggl|\frac{1}5\biggr[f(a)+f(b)+3f\biggl(\frac{a+b}2\biggr)\biggr] -\frac1{b-a}\int_a^bf(x)\td x\biggr|\\
\le\frac{13(b-a)}{200}\biggr[\biggr(\frac{58|f'(a)|^q+137|f'(b)|^q}{195}\biggr)^{1/q}
+\biggr(\frac{137|f'(a)|^q+58|f'(b)|^q}{195}\biggr)^{1/q}\biggr],
\end{multline}
\begin{multline}\label{Qi-Xi-Sep-2011-cor1-ineq-3.32-2}
\biggl|\frac{1}5\biggr[2[f(a)+f(b)]+f\biggl(\frac{a+b}2\biggr)\biggr] -\frac1{b-a}\int_a^bf(x)\td x\biggr|\\*
\le\frac{17(b-a)}{200}\biggr[\biggr(\frac{13|f'(a)|^q+72|f'(b)|^q}{85}\biggr)^{1/q}
+\biggr(\frac{72|f'(a)|^q+13|f'(b)|^q}{85}\biggr)^{1/q}\biggr],
\end{multline}
and
\begin{multline}\label{Qi-Xi-Sep-2011-cor1-ineq-3.33-2}
\biggl|\frac{1}6\biggr[f(a)+f(b)+4f\biggl(\frac{a+b}2\biggr)\biggr] -\frac1{b-a}\int_a^bf(x)\td x\biggr|\\*
\le\frac{5(b-a)}{72}\biggr[\biggr(\frac{29|f'(a)|^q+61|f'(b)|^q}{90}\biggr)^{1/q}
+\biggr(\frac{61|f'(a)|^q+29|f'(b)|^q}{90}\biggr)^{1/q}\biggr].
\end{multline}
\end{cor}

\begin{proof}
This follows from letting $(m, \ell)=(1,0)$, $(2,1)$, $(3,1)$, $(4,1)$, $(5,1)$, $(5,2)$, and $(6,1)$ in~\eqref{Qi-Xi-Sep-2011-cor1-ineq-3.12--p=1}, respectively.
\end{proof}

\begin{cor}\label{Qi-Xi-Sep-2011-Cor-3.1.7}
Let $f:I\subseteq\mathbb{R}\to\mathbb{R}$ be a differentiable function on $I^\circ$, $a,b\in I$ with $a<b$, and $f'\in L[a,b]$. If $\lvert f'(x)\rvert$ is convex on $[a,b]$, then
\begin{align}
\label{Qi-Xi-Sep-2011-cor1-ineq-3.27}
\biggl|f\biggl(\frac{a+b}2\biggr) -\frac1{b-a}\int_a^bf(x)\td x\biggr|&\le\frac{b-a}{8}(|f'(a)|+|f'(b)|),\\
\label{Qi-Xi-Sep-2011-cor1-ineq-3.28}
\biggl|\frac{f(a)+f(b)}2 -\frac1{b-a}\int_a^bf(x)\td x\biggr|&\le\frac{b-a}{8}(|f'(a)|+|f'(b)|),\\
\label{Qi-Xi-Sep-2011-cor1-ineq-3.29}
\biggl|\frac{1}3\biggr[f(a)+f(b)+f\biggl(\frac{a+b}2\biggr)\biggr] -\frac1{b-a}\int_a^bf(x)\td x\biggr|&\le\frac{5(b-a)}{72}[|f'(a)|^q+|f'(b)|],\\
\label{Qi-Xi-Sep-2011-cor1-ineq-3.30}
\biggl|\frac12\biggl[\frac{f(a)+f(b)}2 +f\biggl(\frac{a+b}2\biggr)\biggr] -\frac1{b-a}\int_a^bf(x)\td x\biggr|&\le\frac{b-a}{16}(|f'(a)|+|f'(b)|),\\
\label{Qi-Xi-Sep-2011-cor1-ineq-3.31}
\biggl|\frac{1}5\biggr[f(a)+f(b)+3f\biggl(\frac{a+b}2\biggr)\biggr] -\frac1{b-a}\int_a^bf(x)\td x\biggr|
&\le\frac{13(b-a)}{200}(|f'(a)|+|f'(b)|),\\
\label{Qi-Xi-Sep-2011-cor1-ineq-3.32}
\biggl|\frac{1}5\biggr[2[f(a)+f(b)]+f\biggl(\frac{a+b}2\biggr)\biggr] -\frac1{b-a}\int_a^bf(x)\td x\biggr|
&\le\frac{17(b-a)}{200}(|f'(a)|+|f'(b)|),\\
\label{Qi-Xi-Sep-2011-cor1-ineq-3.33}
\biggl|\frac{1}6\biggr[f(a)+f(b)+4f\biggl(\frac{a+b}2\biggr)\biggr] -\frac1{b-a}\int_a^bf(x)\td x\biggr|
&\le\frac{5(b-a)}{72}(|f'(a)|+|f'(b)|).
\end{align}
\end{cor}
\begin{proof}
These inequalities follow from letting $q=1$ in Corollary~\ref{Qi-Xi-Sep-2011-Cor-3.1.6}.
\end{proof}

\section{Applications to special means}

For two positive numbers $a>0$ and $b>0$, define
\begin{gather}
\begin{aligned}
  A(a,b)&=\frac{a+b}2,&  G(a,b)&=\sqrt{ab}\,,&  H(a,b)&=\frac{2ab}{a+b},
\end{aligned}\\
\begin{aligned}
  I(a,b)&=\begin{cases}
    \dfrac1e\biggl(\dfrac{b^b}{a^a}\biggr)^{1/(b-a)}, &a\ne b,\\
    a,&a=b,
  \end{cases}\\
  L(a,b)&=\begin{cases}
    \dfrac{b-a}{\ln b-\ln a},&a\ne b,\\a,& a=b,
  \end{cases}
\end{aligned}
\end{gather}
and
\begin{equation}
  L_s(a,b)=\begin{cases}
    \biggl[\dfrac{b^{s+1}-a^{s+1}}{(s+1)(b-a)}\biggr]^{1/s}, & \text{$s\ne0,-1$ and $a\ne b$,}\\
    L(a,b),&\text{$s=-1$ and $a\ne b$,}\\
    I(a,b), &\text{$s=0$ and $a\ne b$,}\\
    a,&a=b.
  \end{cases}
\end{equation}
It is well known that $A$, $G$, $H$, $L=L_{-1}$, $I=L_0$, and $L_s$ are respectively called the arithmetic, geometric, harmonic, logarithmic, exponential, and generalized logarithmic means of two positive number $a$ and $b$.

\begin{thm}\label{thm-4.1}
Let $b>a>0$, $q>1$, $q\ge p>0$, $m>0$, $m\ge2\ell\ge0$, and $s\in\mathbb{R}$.
\begin{enumerate}
  \item
If either $s>1$ and $(s-1)q\ge1$ or $s<1$ and $s\ne0$, then
\begin{multline}
  \biggl|\frac{2\ell A\bigl(a^s,b^s\bigr)+(m-2\ell)[A(a,b)]^s}{m}-[L_s(a,b)]^s\biggr| \le\frac{b-a}{4m^2}|s|\biggl(\frac{q-1}{2q-p-1}\biggr)^{1-1/q}\\
\begin{aligned}
  &\times\biggl[\frac1{2m(p+1)(p+2)}\biggr]^{1/q} \bigl[(2\ell)^{(2q-p-1)/(q-1)}+(m-2\ell)^{(2q-p-1)/(q-1)}\bigr]^{1-1/q}\\
  &\times\Bigl\{\bigl[\bigl((2\ell)^{p+2}+(mp+m+2\ell)(m-2\ell)^{p+1}\bigr)a^{(s-1)q} \\
  &+\bigl((2mp+4m-2\ell)(2\ell)^{p+1}+(mp+3m-2\ell)(m-2\ell)^{p+1}\bigr)b^{(s-1)q}\bigr]^{1/q}\\
&+\bigl[\bigl((2mp+4m-2\ell)(2\ell)^{p+1}+(mp+3m-2\ell)(m-2\ell)^{p+1}\bigr)a^{(s-1)q} \\
&+\bigl((2\ell)^{p+2}+(mp+m+2\ell)(m-2\ell)^{p+1}\bigr)b^{(s-1)q}\bigr]^{1/q}\Bigr\};
\end{aligned}
\end{multline}
\item
If $s=-1$, then
\begin{multline}
  \biggl|\frac1m\biggl[\frac{2\ell}{H(a,b)}+\frac{m-2\ell}{A(a,b)}\biggr]-\frac1{L(a,b)}\biggr| \le\frac{b-a}{4m^2}\biggl(\frac{q-1}{2q-p-1}\biggr)^{1-1/q}\biggl[\frac1{2m(p+1)(p+2)}\biggr]^{1/q}\\
  \begin{aligned}
  &\times\bigl[(2\ell)^{(2q-p-1)/(q-1)}+(m-2\ell)^{(2q-p-1)/(q-1)}\bigr]^{1-1/q}\\ &\times\biggl\{\biggl[\frac{(2\ell)^{p+2}+(mp+m+2\ell)(m-2\ell)^{p+1}}{a^{2q}}\\
  &+\frac{(2mp+4m-2\ell)(2\ell)^{p+1}+(mp+3m-2\ell)(m-2\ell)^{p+1}}{b^{2q}}\biggr]^{1/q}\\
  &+\biggl[\frac{(2mp+4m-2\ell)(2\ell)^{p+1}+(mp+3m-2\ell)(m-2\ell)^{p+1}}{a^{2q}}\\
  &+\frac{(2\ell)^{p+2}+(mp+m+2\ell)(m-2\ell)^{p+1}}{b^{2q}}\biggr]^{1/q}\biggr\}.
  \end{aligned}
\end{multline}
\end{enumerate}
\end{thm}

\begin{proof}
Set $f(x)=x^s$ for $x>0$ and $s\ne0,1$. Then it is easy to obtain that
\begin{align*}
f'(x)&=sx^{s-1},& |f'(x)|^q&=|s|^qx^{(s-1)q}, & \bigl(|f'(x)|^q\bigr)''&=(s-1)q[(s-1)q-1]|s|^qx^{(s-1)q-2}.
\end{align*}
Hence, when $s>1$ and $(s-1)q\ge1$, or when $s<1$ and $s\ne0$, the function $|f'(x)|^q$ is convex on $[a,b]$. From Corollary~\ref{Qi-Xi-Sep-2011-Cor-3.1.2}, Theorem~\ref{thm-4.1} follows.
\end{proof}

By the argument similar to Theorem~\ref{thm-4.1}, we further obtain the following conclusions.

\begin{thm}
Let $b>a>0$, $q\ge1$, $m>0$, $m\ge2\ell\ge0$, and $s\in\mathbb{R}$.
\begin{enumerate}
  \item
  If either $s>1$ and $(s-1)q\ge1$ or $s<1$ and $s\ne0$, then
\begin{multline}
    \biggl|\frac{2\ell A(a^s,b^s)+(m-2\ell)[A(a,b)]^s}m-[L_s(a,b)]^s\biggr| \le\frac{b-a}{8m^2}\biggl(\frac1{3m}\biggr)^{1/q}\bigl[4\ell^2+(m-2\ell)^2\bigr]^{1-1/q}|s|\\
    \begin{aligned}
    &\times\Bigl\{\bigl[\bigl(4\ell^3+(m+\ell)(m-2\ell)^2\bigr)a^{(s-1)q} +\bigl(4(3m-\ell)\ell^2+(2m-\ell)(m-2\ell)^2\bigr)b^{(s-1)q}\bigr]^{1/q}\\
    &+\bigl[\bigl(4(3m-\ell)\ell^2+(2m-\ell)(m-2\ell)^2\bigr)a^{(s-1)q} +\bigl(4\ell^3+(m+\ell)(m-2\ell)^2\bigr)b^{(s-1)q}\bigr]^{1/q}\Bigr\}.
    \end{aligned}
  \end{multline}
and
\begin{multline}
    \biggl|\frac{2\ell A(a^s,b^s)+(m-2\ell)[A(a,b)]^s}m-[L_s(a,b)]^s\biggr| \le\frac{b-a}{4m}|s|\biggl[\frac1{2m^2(q+1)(q+2)}\biggr]^{1/q}\\
    \begin{aligned}
    &\times\Bigl\{\bigl[\bigl((2\ell)^{q+2}+(mq+m+2\ell)(m-2\ell)^{q+1}\bigr) a^{(s-1)q}\\
    &+\bigl((2mq+4m-2\ell)(2\ell)^{q+1}+(mq+3m-2\ell)(m-2\ell)^{q+1}\bigr)b^{(s-1)q}\bigr]^{1/q} \\ &+\bigl[\bigl((2mq+4m-2\ell)(2\ell)^{q+2}+(mq+3m-2\ell)(m-2\ell)^{q+1}\bigr) a^{(s-1)q}\\
    &+\bigl((2\ell)^{q+1}+(mq+m+2\ell)(m-2\ell)^{q+1}\bigr)b^{(s-1)q}\bigr]^{1/q}\Bigr\}
    \end{aligned}
  \end{multline}
  \item
  If $s=-1$, then
\begin{multline}
    \biggl|\frac1m\biggl[\frac{2\ell}{H(a,b)}+\frac{m-2\ell}{A(a,b)}\biggr]-\frac1{L(a,b)}\biggr| \le\frac{b-a}{8m^2}\biggl(\frac1{3m}\biggr)^{1/q}\bigl[4\ell^2+(m-2\ell)^2\bigr]^{1-1/q}\\
\begin{aligned}
&\times\biggl\{\biggl[\frac{4\ell^3+(m+\ell)(m-2\ell)^2}{a^{2q}} +\frac{4(3m-\ell)\ell^2+(2m-\ell)(m-2\ell)^2}{b^{2q}}\biggr]^{1/q}\\
&+\biggl[\frac{4(3m-\ell)\ell^2+(2m-\ell)(m-2\ell)^2}{a^{2q}} +\frac{4\ell^3+(m+\ell)(m-2\ell)^2}{b^{2q}}\biggr]^{1/q}\biggr\}.
\end{aligned}
\end{multline}
and
\begin{multline}
  \biggl|\frac1m\biggl[\frac{2\ell}{H(a,b)}+\frac{m-2\ell}{A(a,b)}\biggr]-\frac1{L(a,b)}\biggr| \le\frac{b-a}{4m}\biggl[\frac1{2m^2(q+1)(q+2)}\biggr]^{1/q}\\
  \begin{aligned}
    &\times\Biggl\{\biggl[\frac{(2\ell)^{q+2}+(mq+m+2\ell)(m-2\ell)^{q+1}}{a^{2q}} \\
    &+\frac{(mq+3m-2\ell)(m-2\ell)^{q+1}+(2mq+4m-2\ell)(2\ell)^{q+1}}{b^{2q}}\biggr]^{1/q}\\
    &+\biggl[\frac{(2mq+4m-2\ell)(2\ell)^{q+1}+(mq+3m-2\ell)(m-2\ell)^{q+1}}{a^{2q}} \\
    &+\frac{(2\ell)^{q+2}+(mq+m+2\ell)(m-2\ell)^{q+1}}{b^{2q}}\biggr]^{1/q}\Biggr\}
  \end{aligned}
\end{multline}
\end{enumerate}
In particular, we have
\begin{equation}
\biggl|\frac{2\ell A(a^s,b^s)+(m-2\ell)[A(a,b)]^s}m-[L_s(a,b)]^s\biggr| \le\frac{b-a}{4m^2}|s|\bigl[4\ell^2+(m-2\ell)^2\bigr]A\bigl(a^{s-1},b^{s-1}\bigr)
\end{equation}
and
\begin{equation}
\biggl|\frac1m\biggl[\frac{2\ell}{H(a,b)}+\frac{m-2\ell}{A(a,b)}\biggr]-\frac1{L(a,b)}\biggr| \le\frac{b-a}{4m^2}\frac{4\ell^2+(m-2\ell)^2}{H(a^2,b^2)}.
\end{equation}
\end{thm}

\begin{thm}\label{Xi-Qi-4.4=thm}
Let $b>a>0$, $q>1$, $q\ge p>0$, $m>0$, and $m\ge2\ell\ge0$. Then
\begin{multline}
\biggl|\frac{2\ell\ln G(a,b)+(m-2\ell)\ln A(a,b)}m-\ln I(a,b)\biggr| \le\frac{b-a}{4m^2}\biggl(\frac{q-1}{2q-p-1}\biggr)^{1-1/q}\\
\begin{aligned}
&\times\biggl[\frac1{2m(p+1)(p+2)}\biggr]^{1/q} \bigl[(m-2\ell)^{(2q-p-1)/(q-1)}+(2\ell)^{(2q-p-1)/(q-1)}\bigr]^{1-1/q}\\
&\times\biggl\{\biggl[\frac{(2\ell)^{p+2}+(mp+m+2\ell)(m-2\ell)^{p+1}}{a^q} \\
&+\frac{(mp+3m-2\ell)(m-2\ell)^{p+1}+(2mp+4m-2\ell)(2\ell)^{p+1}}{b^q}\biggr]^{1/q}\\
&+\biggl[\frac{(2mp+4m-2\ell)(2\ell)^{p+1}+(mp+3m-2\ell)(m-2\ell)^{p+1}}{a^q} \\
&+\frac{(2\ell)^{p+2}+(mp+m+2\ell)(m-2\ell)^{p+1}}{b^q}\biggr]^{1/q}\biggr\}.
\end{aligned}
\end{multline}
\end{thm}

\begin{proof}
This follows from taking $f(x)=\ln x$ for $x>0$ in~Corollary~\ref{Qi-Xi-Sep-2011-Cor-3.1.2}.
\end{proof}

By the similar argument to Theorem~\ref{Xi-Qi-4.4=thm}, we can obtain the following inequalities.

\begin{thm}
Let $b>a>0$, $q\ge1$, $m>0$, and $m\ge2\ell\ge0$. Then
\begin{multline}
\biggl|\frac{2\ell\ln G(a,b)+(m-2\ell)\ln A(a,b)}m-\ln I(a,b)\biggr| \le\frac{b-a}{4m}\biggl[\frac1{2m^2(q+1)(q+2)}\biggr]^{1/q}\\
\begin{aligned}
&\times\biggl\{\biggl[\frac{(2\ell)^{q+2}+(mq+m+2\ell)(m-2\ell)^{q+1}}{a^q}\\
& +\frac{(mq+3m-2\ell)(m-2\ell)^{q+1}+(2mq+4m-2\ell)(2\ell)^{q+1}}{b^q}\biggr]^{1/q}\\
&+\biggl[\frac{(2mq+4m-2\ell)(2\ell)^{q+1}+(mq+3m-2\ell)(m-2\ell)^{q+1}}{a^q}\\
&+\frac{(2\ell)^{q+2}+(mq+m+2\ell)(m-2\ell)^{q+1}}{b^q}\biggr]^{1/q}\biggr\}
\end{aligned}
\end{multline}
and
\begin{multline}
\biggl|\frac{2\ell\ln G(a,b)+(m-2\ell)\ln A(a,b)}m-\ln I(a,b)\biggr| \le\frac{b-a}{8m^2}\biggl(\frac1{3m}\biggr)^{1/q}\bigl[(m-2\ell)^2+(2\ell)^2\bigr]^{1-1/q}\\
\begin{aligned}
&\times\biggl\{\biggl[\frac{4\ell^3+(m+\ell)(m-2\ell)^2}{a^q} +\frac{(2m-\ell)(m-2\ell)^2+4(3m-\ell)\ell^2}{b^q}\biggr]^{1/q} \\
&+\biggl[\frac{4(3m-\ell)\ell^2+(2m-\ell)(m-2\ell)^2}{a^q} +\frac{4\ell^3+(m+\ell)(m-2\ell)^2}{b^q}\biggr]^{1/q}\biggr\}.
\end{aligned}
\end{multline}
In particular, we have
\begin{equation}
\biggl|\frac{2\ell\ln G(a,b)+(m-2\ell)\ln A(a,b)}m-\ln I(a,b)\biggr| \le\frac{b-a}{4m^2}\frac{4\ell^2+(m-2\ell)^2}{H(a,b)}.
\end{equation}
\end{thm}

\begin{rem}
Some inequalities of Hermite-Hadamard type were also obtained in~\cite{Hadramard-Convex-Xi-Filomat.tex, H-H-Bai-Wang-Qi-2012.tex, Chun-Ling-H-H-s-convex-3-time.tex, Jiang-Hua-Qi-Analysis-Munich.tex, difference-hermite-hadamard.tex, H-H-Shuang-Ye-Munich.tex, Wang-Qi-Xi-Hadamard-IJOPCM.tex, Wang-Ineq-H-H-type-Analysis.tex, Xi-Bai-Qi-Hadamard-2011-AEQMath.tex, Hadramard-Convex-Xi-September-2011.tex, AM-Xi-Wang-Qi.doc, Zhang-Ji-Qi-1205.doc, GA-H-H-Zhang-lematema.tex} by the authors.
\end{rem}

\end{document}